\documentclass[11pt,reqno,a4paper]{amsart}

\usepackage{amssymb,epsfig,graphics}
\usepackage{amsmath, float}
\usepackage[foot]{amsaddr}

\newtheorem{theorem}{Theorem}[section]
\newtheorem{proposition}[theorem]{Proposition}

\newtheorem{corollary}[theorem]{Corollary}
\newtheorem{lemma}[theorem]{Lemma}
\newtheorem{sub-lemma}[theorem]{Sub-Lemma}

\newtheorem{remark}[theorem]{Remark}

 \def\R{\mathbb{R}}

\DeclareMathOperator{\esssup}{ess sup}
\DeclareMathOperator{\osc}{osc}

\DeclareMathOperator{\Leb}{Leb}

\newcommand{\cF}{{\mathcal F}}
\newcommand{\tSigma}{{\widetilde \Sigma}}

\let\eps=\varepsilon
\renewcommand{\limsup}{\mathop{{\overline {\hbox{{\rm lim}}}}}}

\begin{document}
\title{Variance continuity for Lorenz flows}
\author{Wael Bahsoun$\dagger$}
\address{$\dagger$ Department of Mathematical Sciences, Loughborough University,
Loughborough, LE11 3TU, UK}
\email{$\dagger$ W.Bahsoun@lboro.ac.uk}
\author{Ian Melbourne*}
\address{* Mathematics Institute, University of Warwick, Coventry, CV4 7AL, UK}
\email{* I.Melbourne@warwick.ac.uk}
\author{Marks Ruziboev$\ddagger$}
\address{$\ddagger$ Mathematisch Instituut, Universiteit Leiden, Niels Bohrweg 1, 2333 CA Leiden, The Netherlands}
\email{$\ddagger$ m.ruziboev@math.leidenuniv.nl}
\thanks{WB and MR would like to thank The Leverhulme Trust for supporting their research through the research grant RPG-2015-346. The research of IM was supported in part by a European
Advanced Grant StochExtHomog (ERC AdG 320977).
We are very grateful to the referee for important comments that have led to a much more readable explanation of the method in this paper.}

\begin{abstract}
The classical Lorenz flow, and any flow which is close to it in the $C^2$-topology, satisfies a Central Limit Theorem (CLT). We prove that the variance  in the CLT varies continuously.
\end{abstract}
\date{\today}
\maketitle

\markboth{Wael Bahsoun, Ian Melbourne \and Marks Ruziboev}{Variance continuity for Lorenz flows}
\bibliographystyle{plain}

\section{Introduction}
In 1963, Lorenz \cite{Lo} introduced the following system of equations 
\begin{equation}\label{LS}
\begin{cases}
\dot x = -10x+10y \\
\dot y = 28x-y-xz \\
\dot z = -\frac{8}{3}z+xy
\end{cases}
\end{equation}
as a simplified model for atmospheric convection. Numerical simulations performed by Lorenz showed that the above system exhibits sensitive dependence on initial conditions and has a non-periodic ``strange" attractor.  Since then, \eqref{LS} became a basic example of a chaotic deterministic system that is notoriously difficult to analyse. 

A rigorous mathematical framework of similar flows was initiated with the introduction of the so called  geometric Lorenz attractors in \cite{ABS, GW}. The papers \cite{Tu1, Tu} provided a computer-assisted proof that the classical Lorenz attractor in~\eqref{LS} is indeed a geometric Lorenz attractor.  In particular, it is a singularly hyperbolic attractor~\cite{MPP}, namely a nontrivial robustly transitive attracting invariant set containing a singularity (equilibrium 
point).  Moreover, there is a distinguished Sinai-Ruelle-Bowen (SRB) ergodic probability measure, see for example~\cite{APPV}.
Statistical limit laws, in particular the central limit theorem (CLT) for H\"older observables, were first obtained in \cite{HM} for the classical Lorenz attractor, and were shown for general singular hyperbolic attractors in~\cite{AM19}. 
For further background and a complete list of references up to 2010 we refer the reader to the monograph of Ara\'ujo and Pac\'ifico \cite{AP}. 

Let $X_\eps:\R^3\times\R\to\R^3$, $\eps\ge0$, be a continuous family of $C^2$ flows on $\R^3$ admitting a geometric Lorenz attractor with
singularities $x_\eps$ and corresponding SRB measures $\mu_\eps$.
Precise definitions are given in section~\ref{GLA}; 
in particular, the framework includes the classical Lorenz attractor.
By~\cite{AS,BR}, the flows $X_\eps$
are statistically stable: for any continuous $\psi:\R^3\to\R$ 
$$\lim_{\eps\to 0}\int\psi\, d\mu_{\eps}=\int\psi\, d\mu_0.$$
The CLT in~\cite{AM19,HM} states that for fixed $\eps\ge 0$
and $\psi:\R^3\to\R$ H\"older, there exists $\sigma^2=\sigma^2_{X_\eps}(\psi)\ge0$ such that
\begin{equation} \label{eq:CLT}
\frac{1}{\sqrt t}\left(\int_{0}^{t}\psi\circ X_{\eps}(s)\, ds -t\int \psi\, d\mu_\eps\right){\overset{\text{law}}{\longrightarrow}}  \mathcal {N}(0, \sigma^2)
\quad\text{as $t\to\infty$}.
\end{equation}
By~\cite[Section~4.3]{HM}, $\sigma^2$ is typically nonzero.

We prove continuity of the variance, namely that $\eps\mapsto \sigma^2_{X_\eps}(\psi)$ is continuous.  At the same time, we obtain estimates on the dependence of the variance on $\psi$.  We now state
the main result of the paper.
Define $\|\psi\|=\int|\psi|\,d\mu_0+|\psi(x_0)|$.

\begin{theorem}\label{main3} 
Let $\psi,\,\psi':\R^3\to\R$ be
H\"older observables.  Then
\begin{itemize}
\item[(a)]
$\lim_{\eps\to0}\sigma^2_{X_\eps}(\psi)=\sigma^2_{X_0}(\psi)$; and
\item[(b)]
there exists a constant $C>0$ (depending only on the H\"older norms of $\psi$ and $\psi'$) such that
\[
|\sigma^2_{X_0}(\psi)-\sigma^2_{X_0}(\psi')|\le C\|\psi-\psi'\|(1+|\log\|\psi-\psi'\||).
\]
\end{itemize}
\end{theorem}

\begin{remark}  In part (b), we obtain closeness of the variances provided the observables are close in $L^1$ with respect to both the SRB measure and the Dirac point mass at $0$ (provided the individual H\"older norms are controlled).  It is an easy consequence of the methods in this paper that the logarithmic factor can be removed if the norm $\|\;\|$ is replaced by the H\"older norm.
\end{remark}

\begin{remark}
All results in this paper go through without change for continuous families of $C^r$ flows, $r>1$.  We take $r=2$ for notational convenience.
\end{remark}

Our technique is based on first proving variance continuity for the corresponding family of Poincar\'e maps and then passing the result to the family of flows. The main difficulty in passing from maps to flows lies in the fact that the return time function to the Poincar\'e section is unbounded. 
A key step in overcoming this hurdle is to show that 
for any H\"older observable on $\R^3$ vanishing at the singularity, the induced observable for the Poincar\'e map is piecewise H\"older. 
Related results for various classes of discrete time dynamical systems can be found 
in~\cite{BCV,DZ,KKM} using somewhat different methods, but there are no previous results for Lorenz attractors.

\bigskip

The paper is organised as follows. In section \ref{setup} we recall the basic setup and notation associated with (families of) geometric Lorenz 
attractors. 
In section~\ref{sec:normal} we show how to normalise the families of flows to obtain simplified coordinates for the proofs. Section \ref{Pre} contains properties of one dimensional Lorenz maps. Section \ref{app} studies the family of Poincar\'e maps. It starts by showing that the family of maps admit a uniform rate of correlations decay for piecewise H\"older functions, using suitable anisotropic norms. We then use the Green-Kubo formula to show continuity of the variance for the family of  Poincar\'e maps. In section~\ref{varianceflow}  we prove a version of Theorem~\ref{main3} for normalised families and use this to prove Theorem~\ref{main3}.

\section{Geometric Lorenz attractors}\label{setup}

In this section, we recall the basic setup and notation associated with (families of) geometric Lorenz attractors.
In subsections~\ref{GLA} and~\ref{family}, we recall the class of (families of) geometric Lorenz attractors considered in the paper.

We begin with some notational preliminaries.
Let $U\subset\R^m$ be open.
Fix $\alpha\in(0,1)$ and recall that $f:U\to\R^n$ is $C^\alpha$ if 
there exists $C>0$ such that
$|f_j(x)-f_j(y)|\le C|x-y|^\alpha$ for all $x,y\in U$ 
and all $j=1,\dots,n$.
(Here $|x|=\sqrt{x_1^2+\dots+x_m^2}$ denotes the Euclidean norm on $\R^m$.)
Let $H_{\alpha}(f)$ be the least such constant $C$ and define the H\"older norm
$\|f\|_{\alpha}=|f|_\infty+H_\alpha(f)$.
Then $f$ is $C^{1+\alpha}$ if $Df:U\to\R^{n\times m}$ is $C^\alpha$
and we set 
$\|f\|_{1+\alpha}=|f|_\infty+\|Df\|_{\alpha}$.
A family $f_\eps$ of $C^{1+\alpha}$ maps, $\eps\ge0$, is said to be continuous if $\lim_{\eps\to\eps_0}\|f_\eps-f_{\eps_0}\|_{1+\alpha}=0$.
Similarly, we speak of continuous families of 
$C^2$ flows, $C^{1+\alpha}$ diffeomorphisms, and so on.
In the case of Lorenz flows, we are particularly interested in families of $C^2$ flows on $\R^3$ restricted to an open bounded region $U$ of phase space; for convenience we suppress mentioning the subset $U$.

 \subsection{Definition of geometric Lorenz attractors}  \label{GLA}
There are various notions of geometric Lorenz attractor in the literature depending on the properties being analysed.  Roughly speaking, we take a geometric Lorenz attractor to be a singular hyperbolic attractor for a vector field on $\R^3$ with a single singularity $x_0$  and a connected global cross-section
with a $C^{1+\alpha}$ stable foliation.
As promised, we now give a precise description.

Let $G:\R^3\to\R^3$ be a $C^2$ vector field satisfying $G(x_0)=0$ and let $X$ be the associated flow. 
	We suppose that 
	the differential $DG(x_0)$ at the singularity has three real eigenvalues $\lambda_2<\lambda_3<0<\lambda_1$ satisfying $\lambda_1+\lambda_3>0$ (Lorenz-like singularity).

	Let $\Sigma$ be a two-dimensional rectangular cross-section transverse to the flow chosen in a neighbourhood of the singularity $x_0$,
	and let $\Gamma$ be the intersection of $\Sigma$ with the local stable manifold of $x_0$. 
	We suppose that there exists a well defined Poincar\'e map $F:\Sigma\setminus\Gamma\to\Sigma$.  Moreover, we 
assume that the underlying flow is singular hyperbolic~\cite{MPP}.  It follows~\cite[Theorem~4.2]{AM17} that a neighbourhood of the attractor is foliated by 
one-dimensional $C^2$ stable leaves.  
We assume $q$-bunching for some $q>1$ in~\cite[condition~(4.2)]{AM17}.  By~\cite[Theorem~4.12 and Remark~4.13(b)]{AM17}, it follows that
the stable foliation for the flow is $C^q$.

The foliation by stable leaves for the flow naturally induces (see for example~\cite[Section~3.1]{AM19}) a $C^q$ ($q>1$) foliation inside $\Sigma$ of a neighbourhood of the attractor intersected with $\Sigma$
by one-dimensional $C^2$ stable leaves for the Poincar\'e map $F$.
We denote this stable foliation for $F$ by $\cF$.

The stable leaves for the flow are exponentially contracting~\cite[Theorem~4.2(a)(3)]{AM17} and this property is inherited by the stable leaves for $F$.  This means that there exists $\rho\in(0,1)$, $K>0$ such that
\begin{equation} \label{eq:rho}
		|F^n \xi_1- F^n \xi_2|\le K\rho^n|\xi_1- \xi_2|,
\end{equation}
for $\xi_1,\xi_2$ in the same stable leaf in $\cF$ and $n\ge1$.

	Let $I\subset\Sigma$ be a smoothly ($C^\infty$) embedded one-dimensional subspace transverse to the stable foliation and let $T:I\to I$ be the one-dimensional map obtained from $F$ by quotienting along stable leaves.  Let $\xi_0$ be the intersection of $I$ with $\Gamma$.  

\begin{proposition}\label{T}
$T$ is a Lorenz-like expanding map.  That is, $T$ is monotone (without loss we take $T$ to be increasing) and piecewise $C^{1+\alpha}$  on $I\setminus\{\xi_0\}$ for some $\alpha\in(0,1)$ with
		a singularity at $\xi_0$ and one-sided limits 
\mbox{$T(\xi_0^+)<0$} and $T(\xi_0^-)>0$.
Also, $T$ is uniformly expanding:
		there are constants \mbox{$c>0$} and $\theta>1$ such that $(T^n)'(x)\ge c\theta^{n}$ for all $n\ge1$, whenever $x\notin \bigcup_{j=0}^{n-1}T^{-j}(\xi_0)$.
\end{proposition}

\begin{proof}
The map $F$ is piecewise $C^{1+\alpha}$ and the foliation by stable leaves is $C^{1+\alpha}$, so
$T$ is piecewise $C^{1+\alpha}$ on $I\setminus\{\xi_0\}$.  Uniform expansion follows from~\cite[Theorem~4.3]{AM19}.
The remaining properties are immediate.
\end{proof}

The final part of the definition of geometric Lorenz attractor is that the one-dimensional map $T$ is transitive on the interval 
$[T(\xi_0^+),T(\xi_0^-)]$.
It is then standard~\cite{AS,AM19,APPV,K} that $T$ has a unique absolutely continuous invariant probability measure (acip) $\bar\mu$ leading to a unique SRB measure $\mu$ for the 
geometric Lorenz attractor containing $x_0$. 
The basin of $\mu$ has full Lebesgue measure in a neighbourhood of the 
attractor.

\begin{remark}
The classical Lorenz attractor for the system of equations~\eqref{LS} (and for nearby equations) is an example of a geometric Lorenz attractor as defined above.
Except for $q$-bunching, the assumptions above are verified in~\cite{Tu}.
The $q$-bunching condition is verified in~\cite[Lemma~2.2]{AMV15}.
(By~\cite[Section~5]{AM17}, the optimal value of $q$ lies between $1.278$ and $1.705$; hence we have $C^{1+\alpha}$ regularity for the stable foliation as in~\cite{BR} but not $C^2$ regularity as in~\cite{AS}.)
\end{remark}

 \subsection{Families of geometric Lorenz attractors}  \label{family}
Let $X_\eps$, $\eps\ge0$, be  a continuous  family of $C^2$ flows on $\R^3$ admitting a geometric Lorenz attractor as in subsection~\ref{GLA} with singularity $x_\eps$.  
The constants $K$ and $\rho$ in~\eqref{eq:rho} derive from the singular hyperbolic structure which varies continuously under $C^1$ perturbations. Hence $K$ and $\rho$ can be chosen independent of $\eps$.

Making an initial $C^2$ change of coordinates (varying continuously in $\eps$), we can suppose without loss that $x_\eps\equiv0$ and that $\Sigma$, $\Gamma$ and $I$
are given by 
	$
	\Sigma =\left\{(x, y, 1):-1\le x\le 1,\,-\frac 1 2\le   y \le \frac 1 2 \right\}
	$,
	$\Gamma=\{(0, y, 1): -\frac 1 2 \le  y \le  \frac 1 2\}$
	and $I=\{(x, 0, 1): -\frac 1 2 \le  x \le  \frac 1 2\}\cong[-\frac12,\frac12]$ for all $\eps$.
Throughout the paper, when we speak of a continuous family of $C^2$ flows admitting geometric Lorenz attractors, we assume that this initial change of coordinates has been performed.

Define the Poincar\'e return time to $\Sigma$,
\[
\tau_\eps:\Sigma\to(0,\infty), \qquad
\tau_\eps(\xi)=\inf\{t>0:X_\eps(\xi,t)\in\Sigma\}.
\]

\begin{proposition} \label{tau}
The return time $\tau_\eps:\Sigma\setminus\Gamma\to\Sigma$ 
is given by $\tau_\eps(x, y)=-\frac 1 {\lambda_{1,\eps}} \log|x|+\tau_{2,\eps}(x,y)$ where $\tau_{2,\eps}$ is a continuous family of $C^2$ functions.
\end{proposition}

\begin{proof}
Applying the Hartman-Grobman theorem for fixed $\eps$, we can linearise $X_\eps$ in a neighbourhood of $0$ so that the linearised flow is given by $(x,y,z)\mapsto (e^{\lambda_{1,\eps} t}x,e^{\lambda_{2,\eps} t}y,e^{\lambda_{3,\eps} t}z)$.
The time of flight in this neighbourhood is readily calculated in these coordinates to be
$-\frac 1 {\lambda_{1,\eps}} \log|x|$ for $x\in I$ and the same formula holds in the original coordinates.  The remaining flight time $\tau_{2,\eps}$ is a first hit-time for the $C^2$ flow away from the singularity at $0$ and hence is $C^2$.  Since $X_\eps$ is a continuous family of $C^2$ flows, it follows that $\tau_{2,\eps}$ is a continuous family of $C^2$ functions.
\end{proof}

\begin{theorem} \label{Teps}
Let $X_\eps$ be a continuous family of $C^2$ flows on $\R^3$ admitting a geometric Lorenz attractor.  Then 
there exists $\alpha>0$ such that the one-dimensional maps $T_\eps:I\to I$ form a continuous family of piecewise $C^{1+\alpha}$ maps.
\end{theorem}

\begin{proof}
Recall that $T_\eps$ is obtained from the continuous family of piecewise $C^{1+\alpha}$ maps $F_\eps$ by quotienting along the stable foliation.
Our assumption of $q$-bunching ($q>1$) yields continuous families of $C^q$ stable foliations~\cite{Bortolotti19}.
Hence, $T_\eps$ is a continuous family of piecewise $C^{1+\alpha}$ maps.
\end{proof}

\section{Normalised geometric Lorenz attractors}
\label{sec:normal}

Let $X_\eps$, $\eps\ge0$, be  a continuous family of $C^2$ flows on $\R^3$ admitting a geometric Lorenz attractor.
In this section, we show how to normalise the families of flows to obtain simplified coordinates for carrying out the proofs.

Assume that the preliminary $C^2$ change of coordinates in section~\ref{family} has been performed.
Let $F_\eps:\Sigma\to\Sigma$ and $T_\eps:I\to I$ be the associated families of Poincar\'e maps and one-dimensional piecewise expanding maps.
Also, define $\tSigma =\left\{(x, y, 1):-\frac 1 2\le x,  y \le \frac 1 2 \right\}$.

\begin{proposition}   \label{omegaeps}
There exists a continuous family of $C^{1+\alpha}$ diffeomorphisms $\omega_\eps:\Sigma\to\Sigma$  such that
\begin{itemize}
\item[(a)] $\omega_\eps$ restricts to the identity on $I$; and 
\item[(b)] 
vertical lines in $\tSigma$
are transformed under $\omega_\eps$ into 
stable leaves for $F_\eps:\Sigma\to\Sigma$.
\end{itemize}
\end{proposition}

\begin{proof}  
For $\eps$ fixed, this follows by definition of the smoothness of the stable foliation for $F_\eps$.  
(An explicit formula is given in~\cite[Lemma~4.9]{AM17} where $Y$ and $\chi$ should be replaced by $I$ and $\omega_\eps$, and the embedding is the identity.)
Again, we obtain continuous families of $C^{1+\alpha}$ diffeomorphisms by~\cite{Bortolotti19}.
\end{proof}

The change of coordinates $\omega_\eps$ for the Poincar\'e map $F_\eps$ extends to a change of coordinates $\phi_\eps$ for the flow $X_\eps$.  The extension is standard and essentially unique, though heavy on notation.

First, define the transformed Poincar\'e map and return time
\[
\tilde F_\eps=\omega_\eps^{-1}\circ F_\eps\circ \omega_\eps:\tSigma\to\tSigma,
\qquad 
\tilde\tau_\eps=\tau_\eps\circ\omega_\eps:\tSigma\to(0,\infty).
\]

The set $U_\eps=\{(X_\eps(\xi,t):\xi\in\Sigma,\,0\le t\le\tau_\eps(\xi)\}$ defines a neighbourhood of the Lorenz attractor for $X_\eps$.
Note that $X_\eps(\xi,\tau_\eps(\xi))=F_\eps\xi$ for $\xi\in\Sigma$.
Define 
the transformed flow
$\tilde X_\eps:U_\eps\times[0,\infty)\to U_\eps$ given by
$\tilde X_\eps(x,t)=X_\eps(x,t)$ subject to the identifications
$\tilde X_\eps(\xi,\tilde \tau_\eps(\xi))=\tilde F_\eps\xi$ for $\xi\in\tSigma$.

Finally, define $\phi_\eps:U_\eps\to U_\eps$ by
$\phi_\eps(x)=X_\eps(\omega_\eps(\xi),t)$ for $x= X_\eps(\xi,t)\in U_\eps$, where $\xi\in \Sigma$ and $0\le t\le\tilde \tau_\eps(x)$.  
It follows from the definitions that 
$\phi_\eps|_{\Sigma_\eps}\equiv\omega_\eps$ and 
$\tilde X_\eps(x,t)=\phi_\eps^{-1}\circ X_\eps(\phi_\eps(x),t)$.  Hence $\phi_\eps$ is the desired change of coordinates.

\begin{corollary}   \label{phieps}
The changes of coordinates $\phi_\eps:U\to U$
are continuous families of $C^{1+\alpha}$ diffeomorphisms.
\end{corollary}

\begin{proof}
By assumption, $X_\eps$ is a continuous family of $C^2$ flows.
Hence the result follows from Proposition~\ref{omegaeps}.
\end{proof}

\begin{lemma} \label{normaleps}
The transformed 
flow $\tilde X_\eps(x,t)=\phi_\eps^{-1}\circ X_\eps(\phi_\eps(x),t)$ satisfies:
\begin{itemize}
\item[(i)] The Poincar\'e map $\tilde F_\eps:\tSigma\setminus\Gamma\to\tSigma$ 
is a continuous family of piecewise $C^{1+\alpha}$ diffeomorphisms and
has the form
\[
\tilde F_\eps(x,y)=(T_\eps x,g_\eps(x,y)),
\]
where $T_\eps:I\to I$ is the 
family of Lorenz-like expanding maps in Theorem~\ref{Teps}(a).
\item[(ii)]
		$
		|\tilde F_\eps^n\xi_1- \tilde F_\eps^n\xi_2|\le K\rho^n|\xi_1- \xi_2|
		$
		for all $\xi_1=(x,y_1),\,\xi_2=(x,y_2)$ and $n\ge 1,$ 
\item[(iii)] The return time $\tilde\tau_\eps:\tSigma\setminus\Gamma\to (0,\infty)$ 
is given by $\tilde\tau_\eps(x, y)=-\frac 1 {\lambda_{1,\eps}} \log|x|+\tilde\tau_{2,\eps}(x,y)$ where $\tilde\tau_{2,\eps}$ is a continuous family of 
 $C^{1+\alpha}$ functions.
\item[(iv)] The SRB measures $\tilde\mu_\eps$ for $\tilde X_\eps$ are given by
$\tilde\mu_\eps={\phi_{\eps}^{-1}}\!_*\mu_\eps$ and are statistically stable.
\end{itemize}
\end{lemma}

\begin{proof}
(i)
By Proposition~\ref{omegaeps}(b),
$\tilde F_\eps=\omega_\eps^{-1}\circ F_\eps\circ\omega_\eps$ has the form $\tilde F_\eps(x,y)=(T_\eps x,g_\eps(x,y))$. 
By Proposition~\ref{omegaeps}(a), the
maps $T_\eps:I\to I$ are unchanged by this change of coordinates.
Also,  
$\omega_\eps$ is a continuous family of $C^{1+\alpha}$ diffeomorphisms
and $F_\eps$ is a continuous family of piecewise $C^{1+\alpha}$ diffeomorphisms, 
so $\tilde F_\eps$  a continuous family of piecewise $C^{1+\alpha}$ diffeomorphisms.

\vspace{1ex} \noindent (ii)
This is the uniform contraction condition~\eqref{eq:rho} in the new coordinates.

\vspace{1ex} \noindent
(iii) 
By Proposition~\ref{tau} and  Proposition~\ref{omegaeps},
$\tilde\tau_\eps=\tau_\eps\circ\omega_\eps=-\frac 1 {\lambda_{1,\eps}} \log|x|+\tau_{2,\eps}\circ\omega_\eps$ where
$\tilde\tau_{2,\eps}=\tau_{2,\eps}\circ\omega_\eps$ is a continuous family of 
$C^{1+\alpha}$ functions.

\vspace{1ex} \noindent (iv)
By Proposition~\ref{omegaeps}(a), the acip $\bar\mu_\eps$ for $T_\eps$ is unchanged by the change of coordinates and hence remains absolutely continuous.
Using this and the construction of the SRB measure (see for example~\cite{AS,APPV} for the standard construction of $\mu_\eps$ from $\bar\mu_\eps$),
we obtain that 
$\tilde\mu_\eps={\phi_\eps^{-1}}\!_*\mu_\eps$ is the SRB measure for $\tilde X_\eps$.
Moreover, strong statistical stability~\cite{AS, BR,GL} of the acips $\bar\mu_\eps$ on $I$ is preserved and hence the SRB measures $\tilde\mu_\eps$ remain statistically stable.  
\end{proof}

In the following sections, we prove Theorem~\ref{main3} for normalised families of geometric Lorenz attractors.  
At the end of section~\ref{varianceflow}, we show how results for normalised families imply 
Theorem~\ref{main3}.

\section{The family of one-dimensional maps}\label{Pre}

In this section, we recall some functional-analytic properties associated with the family of one-dimensional Lorenz maps $T_\eps$. 
For $p\ge 1$, we say $f:I \to \R$ is of (universally) bounded $p$-variation if 
	$$
	V_p(f)=
	\sup_{-\frac12= x_0<\dots<x_n= \frac12}
	\left(
	\sum_{i=1}^{n}\left| f(x_i)-f(x_{i-1})\right|^p
			\right)^{1/p} <\infty.
	$$ 
We take $p\ge\frac{1}{\alpha}$.

Let $S_\rho(x)=\{y\in I: |x-y|<\rho\}$.  For $f:I\to \R$, define
$$
 \osc(f, \rho, x)= \esssup\{|f(y_1)-f(y_2)|: y_1, y_2\in S_\rho(x)\}, 
 $$ 
and
$$ \osc_1(f, \rho)=\|\osc(f, \rho, x)\|_1,$$
where the essential supremum is taken with respect to Lebesgue measure on $I\times I$ and $\|\cdot\|_1$ is the $L^1$- norm  with respect to Lebesgue measure on $I.$ Fix $\rho_0>0$ and let $BV_{1, 1/p}\subset L^1$ be the Banach space equipped with the norm
	$$
	\|f\|_{1, 1/p}=V_{1, 1/p}(f)+\|f\|_1,
	\quad\text{where }\quad
	 V_{1, 1/p}(f)=\sup_{0<\rho\le \rho_0}\frac{\osc_1(f, \rho)}{\rho^{1/p}}.
	$$
(The space $BV_{1, 1/p}$ does not depend on $\rho_0$.) 
The fact that  $BV_{1, 1/p}$ is a Banach space is proved in \cite{K}. Moreover, it is proved in \cite{K} that $BV_{1, 1/p}$ is embedded in $L^\infty$ and compactly embedded in $L^1$.
In addition \cite{K} shows that
\begin{equation} \label{VV}
V_{1,1/p}(f)\le 2^{1/p}V_p(f).
\end{equation}

We recall results from the literature that we use later in sections \ref{app} and~\ref{varianceflow} of the paper. Recall from~\cite{K} that $T_\eps$ admits a unique acip $\bar\mu_\eps$ for each $\eps\ge0$. Let $h_\eps$ denote the density for $\bar\mu_\eps$. Let
$P_{\eps}:L^1(I)\to L^1(I)$
denote the transfer operator (Perron-Frobenius) associated with $T_\eps$
(so $\int P_\eps f\,g\,d\Leb=\int f\,g\circ T_\eps\,d\Leb$ for $f\in L^1(I)$, $g\in L^\infty(I)$).

\begin{proposition}\label{recap}\text{ }
There exists $C>0$ and $\Lambda\in(0,1)$ such that 
\begin{itemize}
\item[(a)] $||h_\eps||_\infty<C$, and
\item[(b)] 
$\| P^n_\eps f-h_\eps\int f\,d\Leb\|_{1,1/p}\le C\Lambda^n||f||_{1, 1/p}$
\end{itemize}
for all $n\ge1$, $\eps\ge0$, $f\in BV_{1,1/p}$.
\end{proposition}

\begin{proof}
By~\cite[Lemma~3.3]{BR}, there exist $A_1, A_2>0$, $0<\kappa<1$, such that for all $n\ge1$, $f\in BV_{1, 1/p}$,
$$||P^n_\eps f||_{1,1/p}\le A_1\kappa^n||f||_{1, 1/p}+A_2||f||_1.$$
Taking $f=h_\eps$ and letting $n\to\infty$, we obtain 
$\|h_\eps\|_{1,1/p}\le A_2$ and
part~(a) follows.
Moreover, it was proved in \cite{BR} that
$$\lim_{\eps\to 0}\sup_{||f||_{1,1/p}\le 1}||(P_\eps-P_0)f||_1=0.$$
Thus, the Keller-Liverani stability result \cite{KL} implies that $P_\eps$ has a uniform (in $\eps$) spectral gap on $BV_{1,1/p}$. This proves part~(b). 
\end{proof}

\section{Variance continuity for the Poincar\'e maps}\label{app}

In this section, we prove the analogue of Theorem~\ref{main3} at the level of the Poincar\'e maps $F_\eps$.
Throughout, we work with normalised families as in section~\ref{sec:normal}.

It is well known~\cite{APPV} that $F_\eps$ admits a unique SRB measure $\mu_{F_\eps}$ for each $\eps\ge0$. Moreover, for any continuous $\psi:\Sigma\to\R$ we have
$$\lim_{\eps\to 0}\int\psi\, d\mu_{F_\eps}=\int\psi\, d\mu_{F_0};$$
 i.e.\ the family $F_\eps$ is statistically stable \cite{AS, BR,GL}. 
We require the following strengthening of this property.
In general, we say that $\Psi:\Sigma\to\R$ is piecewise continuous if it is uniformly continuous on the connected components of $\Sigma\setminus\Gamma$.  
Similarly
$\Psi$ is piecewise H\"older if it is uniformly H\"older on the connected components of $\Sigma\setminus\Gamma$.  

\begin{proposition} \label{ss}
Suppose that $\Psi:\Sigma\to\R$ is piecewise continuous and fix $n\ge0$.
Then $\lim_{\eps\to0}\int \Psi\cdot(\Psi\circ F_0^n)\,d\mu_{F_\eps}=\int\Psi\cdot(\Psi\circ F_0^n)\,d\mu_{F_0}$.
\end{proposition}
\begin{proof} 
If $\Psi\cdot(\Psi\circ F_0^n)$ were continuous, this would be immediate from the statement of Proposition 3.3 in~\cite{AS}.  The proof in~\cite{AS} already accounts for trajectories that visit $\Gamma$ in finitely many steps, and it is easily checked that the same arguments apply here.
\end{proof}

For fixed $\eps\ge0$, the CLT holds (see for instance \cite{MN}); i.e.\ for $\Psi:\Sigma\to\R$ piecewise H\"older there exists
$\sigma^2=\sigma^2_{F_\eps}(\Psi)$ such that
$$\frac{1}{\sqrt{n}}\left(\sum_{i=0}^{n-1}\Psi\circ F_\eps^i-n\int\Psi\, d\mu_{F_\eps}\right){\overset{\text{law}}{\longrightarrow}}\mathcal N(0, \sigma^2)
\quad\text{as $n\to\infty$}.$$
The variance  
satisfies the Green-Kubo formula 
\begin{equation}\label{Kubo}
\sigma^2=\int\hat\Psi_\eps^2\,d\mu_{F_\eps}+2\sum_{n=1}^{\infty}\int\hat\Psi_\eps \cdot (\hat\Psi_\eps\circ F^n_{\eps})\, d\mu_{F_\eps},
\end{equation}
where
$\hat\Psi_\eps=\Psi_\eps-\int\Psi_\eps\, d\mu_{F_\eps}$.

\subsection{Uniform decay of correlations}
Let $\alpha\in(0,1]$ and 
$p\ge \frac{1}{\alpha}$.
Let $\Psi:\Sigma\to\R$ be piecewise H\"older with exponent $\alpha$. 
Set 
$$||\Psi||_{\alpha,s}=H_{\alpha,s}(\Psi)+||\Psi||_{\infty}, \qquad
H_{\alpha,s}(\Psi)=\sup_{\underset{y_1\not=y_2}{x,y_1,y_2\in I}}\frac{|\Psi(x,y_2)-\Psi(x,y_1)|}{|y_2-y_1|^\alpha}.$$

For $-\frac12=x_0< \dots< x_n=\frac12$ and $y_i\in I$, $1\le i\le n$, we define
$$\hat V_p(\Psi;(x_0,\dots,x_n); (y_1,\dots,y_n))= \Big(\sum_{1\le i\le n}|\Psi(x_{i-1},y_i)-\Psi(x_{i}, y_i)|^p\Big)^{1/p},$$
and 
$$\hat V_p(\Psi)=\sup \hat V_p(\Psi;(x_0,\dots,x_n); (y_1,\dots,y_n)),$$
where the supremum is taken over all finite partitions of $I$ and all choices of $y_i\in I$. Finally, we define
$$(\Pi \Psi)(x)=\int_I\Psi(x,y)\,dy.$$
We state and prove the following Theorem about uniform (in $\eps$) decay of correlations. 
Define  
$$
D_{\Psi}=||\Pi\Psi||_{1,1/p}+\hat V_p(\Psi)+||\Psi||_{\alpha,s}.
$$
Note that $D_\Psi$ is finite.

\begin{theorem}\label{GalatoloThm} Assume that there exists $K>0$ such that 
\begin{equation}\label{GalatoloAss}
||\Pi(\Psi\circ F_\eps^j)||_{1,1/p}+\hat V_p(\Psi\circ F_\eps^j)\le  D_\Psi K^j 
\end{equation}
for all $j\ge 1$, $\eps\ge0$.
Then there exist $C>0$ and $\theta\in(0,1)$ such that 
$$
\begin{aligned}
&\left|\int\Psi\cdot(\Psi\circ F_\eps^n)\, d\mu_{F_\eps}-\Big(\int\Psi\, d\mu_{F_\eps}\Big)^2\right| 
&\le C D_{\Psi}||\Psi||_{\alpha,s} \theta^n 
\end{aligned}
$$
for all $n\ge1$, $\eps\ge0$.
\end{theorem}

\begin{proof}
The proof follows from~\cite[Theorem~3]{AGP}. The fact that $C$ and $\theta$ do not depend on $\eps$ follows from the uniformity of the constants in Proposition \ref{recap}(b) and the uniformity of the contraction on vertical fibres. 
\end{proof}
We now verify assumption \eqref{GalatoloAss} in Theorem \ref{GalatoloThm}.

\begin{lemma}\label{Galatolo1} 
$V_{1,1/p}(\Pi \Psi)\le 2^{1/p}\hat V_p(\Psi)$.
\end{lemma}

\begin{proof}
Fix $y\in I$ and let $x_0<\dots <x_n$ be a partition of $I$. We have
\[
\begin{aligned}
\sum_{1\le i\le n}|\Pi\Psi(x_{i-1})-\Pi\Psi(x_{i})|^p
& = \sum_{1\le i\le n}\Big|\int (\Psi(x_{i-1},y)-\Psi(x_{i}, y))\,dy\Big|^p
\\ & \le \sum_{1\le i\le n}\int |\Psi(x_{i-1},y)-\Psi(x_{i}, y)|^p \,dy
  \le \hat V_p(\Psi)^p.
\end{aligned}
\]
Combining this with~\eqref{VV} we have
$V_{1,1/p}(\Pi\Psi)\le 2^{1/p}V_p(\Pi\Psi)\le 2^{1/p}\hat V_p(\Psi)$ as required.
\end{proof}

Recall that $F_\eps$ can be written in coordinates as $F_\eps(x,y)=(T_\eps x,g_\eps(x,y))$.  

\begin{lemma}\label{Galatolo2}
Let $M=4(1+\sup_\eps\|g_\eps\|_{C^1}^{\alpha})$. 
Then 
$$\hat V_p(\Psi\circ F_\eps^j)\le (2^j-1)M||\Psi||_{\alpha,s}+2^j\hat V_p(\Psi)
\quad\text{for all $j\ge1$, $\eps\ge0$.}$$
\end{lemma}

\begin{proof}
We suppress the dependence on $\eps$.
Fix $-\frac12=x_0\le x_1\le \cdots\le x_n=\frac12$ and $y_i\in I$, $1\le i\le n$.  Then
\begin{align*}
&\hat V_p(\Psi\circ F;(x_0,\dots,x_n); (y_1,\dots,y_n))^p
=\sum_{1\le i\le n}|\Psi\circ F(x_{i-1},y_i)-\Psi\circ F(x_{i}, y_i)|^p\\
&\le2^{p-1}\sum_{1\le i\le n}H_{\alpha,s}(\Psi)^p|g(x_{i-1},y_i)-g(x_{i},y_i)|^{\alpha p}\\ 
& \qquad +2^{p-1}\sum_{1\le i\le n}|\Psi(T x_{i-1},g(x_{i-1},y_i))-\Psi(T x_{i},g(x_{i-1},y_i))|^p\\
&\le 2^{p-1}H_{\alpha,s}(\Psi)^p 
{\|g\|}_{C^1}^{\alpha p} \\
& \qquad + 2^{p-1}\sum_{1\le i\le n}|\Psi(T x_{i-1},g(x_{i-1},y_i))-\Psi(T x_{i},g(x_{i-1},y_i))|^p.
\end{align*}
Let $x_0,\dots, x_k\in[-\frac{1}{2},0)$ and $x_{k+1}\notin [-\frac{1}{2},0)$. Since $T|_{[-\frac{1}{2},0)}$ is continuous and increasing, we get
\begin{align*}
\sum_{1\le i\le k}|\Psi(T x_{i-1},g(x_{i-1},y_i))-\Psi(T x_{i},g(x_{i-1},y_i))|^p\le \hat V_p(\Psi)^p.
\end{align*}
A similar estimate holds for  $x_{k+1},\dots, x_n\in(0,\frac{1}{2}]$. Moreover, 
$$|\Psi(T x_{k},g(x_{k},y_{k+1}))-\Psi(T x_{k+1},g(x_{k},y_{k+1}))|^p\le 2^p||\Psi||_\infty^p.$$
Consequently,
\begin{align*}
\hat V_p(\Psi\circ F)^p & \le
 2^{p-1}\left( H_{\alpha,s}(\Psi)^p 
\|g\|_{C^1}^{\alpha p}
+2\hat V_p(\Psi)^p+ 2^p||\Psi||_\infty^p\right)\\
&\le  M^p||\Psi||_{\alpha,s}^p+2^p\hat V_p(\Psi)^p\le (M||\Psi||_{\alpha,s}+2\hat V_p(\Psi))^p.
\end{align*}
Therefore,
$$\hat V_p(\Psi\circ F)\le M||\Psi||_{\alpha,s}+2\hat V_p(\Psi).$$
Using the last inequality repeatedly and the fact that $||\Psi\circ F||_{\alpha,s}\le ||\Psi||_{\alpha,s}$, we obtain the result.
\end{proof}

For $\Psi:\Sigma\to\R$ piecewise $C^\alpha$, we define
\[
\|\Psi\|_\alpha=\|\Psi\|_\infty+H_\alpha(\Psi), \qquad
H_{\alpha}(\Psi)=\sup_{\xi_1\neq\xi_2}\frac{|\Psi(\xi_2)-\Psi(\xi_1)|}{|\xi_2-\xi_1|^\alpha},
\]
where we restrict to the cases $\xi_1,\xi_2\in\Sigma^+$ and 
$\xi_1,\xi_2\in\Sigma^-$ in the supremum.

\begin{corollary} \label{cor:Galatolo}
There exists $C>0$ and $\theta\in(0,1)$ such that 
$$
\begin{aligned}
&\left|\int\Psi\cdot(\Psi\circ F_\eps^n)\, d\mu_{F_\eps}-\Big(\int\Psi\, d\mu_{F_\eps}\Big)^2\right| 
&\le C ||\Psi||^2_\alpha\, \theta^n 
\end{aligned}
$$
for all $n\ge1$, $\eps\ge0$ and all piecewise $C^\alpha$ observables
$\Psi:\Sigma\to\R$.  
\end{corollary}

\begin{proof}
By Lemma~\ref{Galatolo1},
\[
||\Pi(\Psi\circ F_\eps^j)||_{1,1/p}+\hat V_p(\Psi\circ F_\eps^j)
\le \|\Psi\|_\infty+(2^{1/p}+1)\hat V_p(\Psi\circ F_\eps^j).
\]
By Lemma~\ref{Galatolo2}, there is a constant $K_0>1$ such that
\[
||\Pi(\Psi\circ F_\eps^j)||_{1,1/p}+\hat V_p(\Psi\circ F_\eps^j)
\le K_02^j(\|\Psi\|_{\alpha,s}+\hat V_p(\Psi)).
\]
Hence assumption~\eqref{GalatoloAss} holds with   $K=2K_0$.
The result follows from Theorem~\ref{GalatoloThm} since
$D_\Psi\le 3\|\Psi\|_\alpha$ and
$\|\Psi\|_{\alpha,s}\le \|\Psi\|_\alpha$.
\end{proof}

\subsection{Variance continuity for the family of two dimensional maps}
We continue to fix $\alpha\in(0,1]$.

\begin{theorem}\label{main2}
Let $\Psi_\eps$, $\eps\ge 0$,  be piecewise $C^\alpha$ with $\sup_\eps||\Psi_\eps||_{\alpha}<\infty$. Assume that 
\begin{equation} \label{emain2}
\lim_{\eps\to 0}\int |\Psi_\eps-\Psi_0|\,d\mu_{F_\eps}=0.
\end{equation}
Then $\lim_{\eps\to 0}\sigma_{F_\eps}^2(\Psi_\eps)=\sigma_{F_0}^2(\Psi_0).$
\end{theorem}
\noindent (The hypotheses of this result will be verified in section~\ref{varianceflow}.
In particular, condition~\eqref{emain2} is addressed in Lemma~\ref{Modified_New_Marks}.)

We first prove a lemma that will be used in the proof of Theorem \ref{main2}.

\begin{lemma}\label{Galatolo1.5}
Let $\Psi:\Sigma\to \R$ be piecewise continuous and fix $n\ge0$.
Then
$$\lim_{\eps\to 0}\int |\Psi\circ F_0^n-\Psi\circ F_0^{n-1}\circ F_\eps|\,d\mu_{F_\eps}=0.$$
\end{lemma}
\begin{proof}
Let $\delta>0$, $j\ge0$, and define $E_{\delta,j}=\bigcup_{i=0}^j F_0^{-i}(B_\delta)$ 
where $B_\delta$ is the $\delta$-neighbourhood of $\Gamma$.
Then $\mu_{F_0}(E_{\delta,n})\le (n+1)\mu_{F_0}(B_\delta)\le M\delta$ where
$M=2(n+1)\|h_0\|_\infty$.

The closure of $E_{\delta,n}$ lies in the interior of $E_{2\delta,n}$,
so there exists a continuous function $\chi:\Sigma\to[0,1]$ supported in $E_{2\delta,n}$ and equal to $1$ on $E_{\delta,n}$.
By statistical stability, 
for $\eps$ sufficiently small,
$\mu_{F_\eps}(E_{\delta,n})
\le \int\chi d\mu_{F_\eps}\le 2\int\chi d\mu_{F_0}\le
2\mu_{F_0}(E_{2\delta,n})\le 4M\delta$.

Let $E_{\delta,j}^c=\Sigma\setminus E_{\delta,j}$.
Now $F_0(\xi)\in E_{\delta,n-1}^c$ for $\xi\in E_{\delta,n}^c$.  Also, $F_\eps\to F_0$ uniformly on $E_{\delta,n}^c$ as $\eps\to0$, so  $F_0(\xi),\,F_\eps(\xi)\in E_{\delta/2,n-1}^c$ for all $\xi\in E_{\delta,n}^c$ and all sufficiently small $\eps$.
Moreover, $\Psi\circ F_0^{n-1}$ is uniformly continuous on $E_{\delta/2,n-1}^c$.
It follows from this and the uniform convergence of $F_\eps$ on $E_{\delta,n}^c$ that $S_\eps=\sup_{E_{\delta,n}^c}|\Psi\circ F_0^{n-1}\circ F_\eps-\Psi\circ F_0^n|\to0$ as $\eps\to0$.

Hence
\begin{align*}
\int &  |\Psi\circ F_0^n-\Psi\circ F_0^{n-1}\circ F_\eps|\,d\mu_{F_\eps}
 \le\int_{E_{\delta,n}} |\Psi\circ F_0^n-\Psi\circ F_0^{n-1}\circ F_\eps|\,d\mu_{F_\eps}\\ & \qquad\qquad\qquad\qquad\qquad\qquad\qquad +\int_{E_{\delta,n}^c} |\Psi\circ F_0^n-\Psi\circ F_0^{n-1}\circ F_\eps|\,d\mu_{F_\eps}\\
&\le2||\Psi||_\infty\,\mu_{F_\eps}(E_{\delta,n})+S_\eps 
\le 8||\Psi||_\infty M\delta+S_\eps
\to 8||\Psi||_\infty M\delta.
\end{align*}
The result follows since $\delta>0$ is arbitrary.
\end{proof}

\begin{proof}[Proof of Theorem \ref{main2}]
We use the Green-Kubo formula \eqref{Kubo}. 
Recall that $\hat\Psi_\eps=\Psi_\eps-\int\Psi_\eps\,d\mu_{F_\eps}$.
By Corollary~\ref{cor:Galatolo}, the series in~\eqref{Kubo} is absolutely convergent uniformly in $\eps$.
Therefore, it suffices to show that 
$\int \hat\Psi_\eps \cdot (\hat\Psi_\eps\circ F^n_{\eps})\, d\mu_{F_\eps}\to
\int\hat\Psi_0 \cdot (\hat\Psi_0\circ F_0^n)\, d\mu_{F_0}$ as $\eps\to0$ for each fixed $n\ge0$.  Now
\begin{equation*}
\begin{split}
&\int\hat\Psi_\eps \cdot (\hat\Psi_\eps\circ F^n_{\eps})\, d\mu_{F_\eps}-\int\hat\Psi_0 \cdot (\hat\Psi_0\circ F_0^n)\, d\mu_{F_0}\\
&=\int\hat\Psi_\eps \cdot (\hat\Psi_\eps\circ F^n_{\eps})\, d\mu_{F_\eps}-\int\hat\Psi_0 \cdot (\hat\Psi_\eps\circ F^n_\eps)\, d\mu_{F_\eps}\\
&\qquad +\int\hat\Psi_0 \cdot (\hat\Psi_\eps\circ F^n_\eps)\, d\mu_{F_\eps}-\int\hat\Psi_0 \cdot (\hat\Psi_0\circ F^n_\eps)\, d\mu_{F_\eps}\\
&\qquad +\int\hat\Psi_0 \cdot (\hat\Psi_0\circ F^n_\eps)\, d\mu_{F_\eps}-\int\hat\Psi_0 \cdot (\hat\Psi_0\circ F_0^n)\, d\mu_{F_\eps}\\
&\qquad +\int\hat\Psi_0 \cdot (\hat\Psi_0\circ F_0^n)\, d\mu_{F_\eps}-\int\hat\Psi_0 \cdot (\hat\Psi_0\circ F_0^n)\, d\mu_{F_0}\\
&=(I) + (II) + (III) + (IV).
\end{split}
\end{equation*}
We have $|(II)|\le 
\|\hat\Psi_0\|_\infty\int|\hat\Psi_\eps-\hat\Psi_0|\circ F_\eps^n\,d\mu_{F_\eps}=
\|\hat\Psi_0\|_\infty\int|\hat\Psi_\eps-\hat\Psi_0|\,d\mu_{F_\eps}\to0$ by~\eqref{emain2}, and similarly $(I)\to0$.  Also, $(IV)\to0$ by
Proposition~\ref{ss}. 

Finally, $|(III)|\le \|\hat\Psi_0\|_\infty\int|\hat\Psi_0\circ F_0^n-\hat\Psi_0\circ F_\eps^n|\,d\mu_{F_\eps}$.
Note that
\begin{equation*}
\begin{split}
\int|\hat\Psi_0\circ F_0^n-\hat\Psi_0\circ F_\eps^n|\,d\mu_{F_\eps}&
\le \sum_{i=1}^{n}\int |\hat\Psi_0\circ F_0^i\circ F_\eps^{n-i}-\hat\Psi_0\circ F_0^{i-1}\circ F_\eps^{n-i+1}|\,d\mu_{F_\eps}\\
& = \sum_{i=1}^{n}\int |\hat\Psi_0\circ F_0^i-\hat\Psi_0\circ F_0^{i-1}\circ F_\eps |\,d\mu_{F_\eps}.
\end{split}
\end{equation*}
Thus, by Lemma \ref{Galatolo1.5}, $(III)\to 0$ as $\eps\to 0$.
\end{proof}

We end this section with the following result which gives explicit estimates in terms of $\Psi$ as required for the proof of Theorem~\ref{main3}(b).

\begin{proposition}\label{p:main2}
Let $\|\Psi\|_{1,\eps}= \int |\Psi|\,d\mu_{F_\eps}$.
For all $H>0$ and $\alpha>0$,
there exists  $C>0$ such that 
\[
|\sigma_{F_\eps}^2(\Psi)-\sigma_{F_\eps}^2(\Psi')|\le 
C \|\Psi-\Psi'\|_{1,\eps}(1+ |\log\|\Psi-\Psi'\|_{1,\eps}|)
\]
for all $\Psi$, $\Psi'$ piecewise H\"older with
$\|\Psi\|_\alpha,\,\|\Psi'\|_\alpha\le H$ and all $\eps\ge0$.
\end{proposition}

\begin{proof}
Let $\hat\Psi=\Psi-\int\Psi\,d\mu_{F_\eps}$ and 
$\hat\Psi'=\Psi'-\int\Psi'\,d\mu_{F_\eps}$.
Let $N\ge1$.
It follows from the Green-Kubo formula~\eqref{Kubo} that
\begin{equation*}
\begin{aligned}
& |\sigma_{F_\eps}^2(\Psi)-\sigma_{F_\eps}^2(\Psi')|\le 
2\sum_{n=0}^{N-1}\|\hat\Psi(\hat\Psi\circ F_\eps^n)
-\hat\Psi'(\hat\Psi'\circ F_\eps^n)\|_{1,\eps}
\\ & \qquad +2\sum_{n=N}^{\infty}\|\hat\Psi(\hat\Psi\circ F_\eps^n)\|_{1,\eps}
+2\sum_{n=N}^{\infty}\|\hat\Psi'(\hat\Psi'\circ F_\eps^n)\|_{1,\eps}.
\end{aligned}
\end{equation*}
Now
\begin{equation*}
\begin{aligned}
 & \|\hat\Psi(\hat\Psi\circ F_\eps^n) -\hat\Psi'(\hat\Psi'\circ F_\eps^n)\|_{1,\eps}
 \le \|\hat\Psi-\hat\Psi'\|_{1,\eps}\|\hat\Psi\circ F_\eps^n\|_\infty
\\ & \qquad\qquad\qquad\qquad\qquad\qquad\qquad\qquad +\|\hat\Psi'\|_\infty\|\hat\Psi\circ F_\eps^n-\hat\Psi'\circ F_\eps^n\|_{1,\eps}
\\ & \qquad\qquad = (\|\hat\Psi\|_\infty+\|\hat\Psi'\|_\infty)\|\hat\Psi-\hat\Psi'\|_{1,\eps}
\le 
 4(\|\Psi\|_\infty+\|\Psi'\|_\infty)\|\Psi-\Psi'\|_{1,\eps}.
\end{aligned}
\end{equation*}
Also, by Corollary~\ref{cor:Galatolo},
$\|\hat\Psi(\hat\Psi\circ F_\eps^n)\|_{1,\eps}\le 
C\theta^n\|\hat\Psi\|_\alpha^2
\le 4C\theta^n\|\Psi\|_\alpha^2$ and similarly for $\Psi'$.
Hence
\[
|\sigma_{F_\eps}^2(\Psi)-\sigma_{F_\eps}^2(\Psi')|\le 
8N(\|\Psi\|_\infty+\|\Psi'\|_\infty)\|\Psi-\Psi'\|_{1,\eps}
+C'\theta^N(\|\Psi\|_\alpha^2+\|\Psi'\|_\alpha^2).
\]
where $C'=8C(1-\theta)^{-1}$.
Taking $N=[q|\log\|\Psi-\Psi'\|_{1,\eps}|]$ for $q$ sufficiently large yields the desired result.
\end{proof}

\section{Variance continuity for the flows}\label{varianceflow}

By \cite{AM19,HM}, the CLT holds for H\"older observables for the Lorenz flows $X_\eps$.
In this section, we show how to obtain continuity of the variances, proving Theorem~\ref{main3}.
During most of this section we continue to work with normalised families as in section~\ref{sec:normal}, culminating in Corollary~\ref{maincor} which is an analogue of Theorem~\ref{main3} for normalised families.  We conclude by using
Corollary~\ref{maincor} to prove Theorem~\ref{main3}.
Throughout we fix $\beta\in(0,1)$.

First, we recall some results from the literature.
Let $\psi:\R^3\to\R$ be $C^\beta$.
Define 
\[
\tilde\psi(x)=\psi(x)-\psi(0).
\] 
Also, define the induced observables $\Psi_{\eps}:\Sigma\setminus\Gamma\to \R$, $\eps\ge0$, by 
\begin{equation}\label{inducedobs}
 \Psi_{\eps}(\xi)=\int_{0}^{\tau_{\eps}(\xi)}\tilde\psi(X_\eps(\xi, t))\,dt.
\end{equation}
The left-hand side of~\eqref{eq:CLT} is identical for 
$\psi$ and $\tilde\psi$, so
\begin{equation} \label{eq:MT1}
\sigma^2_{X_\eps}(\psi)=\sigma^2_{X_\eps}(\tilde\psi).
\end{equation}

\begin{proposition} \label{MelbTor}
The variances for $\tilde\psi$ and $\Psi_{\eps}$ are related by
\[
\sigma^2_{X_\eps}(\tilde\psi)=\frac{\sigma_{F_\eps}^2( \Psi_{\eps})}{\int \tau_\eps\, d\mu_{F_\eps}}.
\]
\end{proposition}

\begin{proof}
By~\cite{HM}, $\tilde\psi$ satisfies the CLT.
The method of proof in~\cite{HM} is to show that 
$\Psi_{\eps}$ and $\tau_\eps$ satisfy the CLT, whereupon the CLT for $\tilde\psi$ follows from~\cite[Theorem~1.1]{MelbTor}.  The desired relation between the variances is given explicitly in~\cite[Theorem~1.1]{MelbTor}.
\end{proof}

By~\cite[Lemma~4.1]{AS},
\begin{equation}\label{tauntau}
\lim_{\eps \to 0}\int \tau_\eps\, d\mu_{F_\eps}=\int \tau_0\, d\mu_{F_0}.
\end{equation}
Hence the main quantity to control is $\sigma_{F_\eps}^2( \Psi_{\eps})$.
First, we verify condition~\eqref{emain2} in Theorem~\ref{main2}.

\begin{lemma}\label{Modified_New_Marks}
$\lim_{\eps\to0}\int|\Psi_{\eps}-\Psi_{0}|\,d\mu_{F_\eps}=0$.
\end{lemma}

\begin{proof}
Fix $N\ge1$ and set $\tau_{\eps, N}(\xi)=\min\{\tau_{ \eps}(\xi), N\}$.
Define
$$ \Psi_{\eps, N}(\xi)=\int_{0}^{\tau_{\eps, N}(\xi)}\tilde\psi(X_\eps(\xi, t))\,dt.$$
Then 
$$
|\Psi_0-\Psi_{\eps}|\le |\Psi_0-\Psi_{0,N}|+|\Psi_{0,N}- \Psi_{\eps, N}|+|\Psi_{\eps, N}-\Psi_{\eps}|.
$$

Now,
$$
|\Psi_\eps(\xi)-\Psi_{\eps, N}(\xi)|=\Big|\int_{\tau_{\eps, N}(\xi)}^{\tau_{\eps}(\xi)}\tilde\psi(X_\eps(\xi, t))\,dt\Big| 
\le 2\|\psi\|_{\infty}(\tau_{\eps}(\xi)-\tau_{\eps, N}(\xi)).
$$
Recall that $\xi=(x,y,1)$.
Now $\tau_\eps(\xi) -\tau_{\eps,N}(\xi)\le -C\log|x|-N$, and $\tau_\eps-\tau_{\eps,N}$ is supported on $B_\eps=\{\xi:\tau_\eps(\xi) >N\}\subset \{\xi: |x|\ <e^{-N/C}\}$. Hence, 
 \begin{equation*}
 \begin{aligned}
\int & |\Psi_\eps-\Psi_{\eps, N}|\,d\mu_{F_\eps} 
 \le 2\|h_\eps\|_\infty\|\psi\|_\infty\int_{B_\eps}(-C\log|x|-N)\,dx
\\ & = 4\|h_\eps\|_\infty\|\psi\|_\infty\int_0^{e^{-N/C}}(-C\log x-N)\,dx
 =4C\|h_\eps\|_\infty{||\psi||}_{\infty} e^{-N/C}.
\end{aligned}
\end{equation*}
By Proposition~\ref{recap}(a), there exists $C'>0$ such that
$\int |\Psi_\eps-\Psi_{\eps, N}|\,d\mu_{F_\eps} \le C'e^{-N/C}\|\psi\|_\infty$.
Similarly,
$\int |\Psi_0-\Psi_{0, N}|\,d\mu_{F_\eps} \le C'e^{-N/C}\|\psi\|_\infty$.

Next,
$|\Psi_{\eps,N}(\xi)-\Psi_{0,N}(\xi)| \le (I)+(II)$
where
\begin{equation*}
\begin{aligned}
(I) & =
\int_{\tau_{0,N}(\xi)}^{\tau_{\eps,N}(\xi)}|\tilde\psi(X_\eps(\xi,t))|\,dt
\le 2\|\psi\|_\infty|\tau_{\eps,N}(\xi)-\tau_{0,N}(\xi)|,
\\ (II) & =
\int_0^{\tau_{0,N}(\xi)}|\tilde\psi(X_\eps(\xi,t))-\tilde\psi(X_0(\xi,t))|\,dt
 \\ & 
\!\!\! 
\!\!\! 
\le 
\int_0^N|\psi(X_\eps(\xi,t))-\psi(X_0(\xi,t))|\,dt
 \le 
H_\beta(\psi)\int_0^N|X_\eps(\xi,t)-X_0(\xi,t)|^\beta \,dt.
\end{aligned}
\end{equation*}

For $\delta>0$ fixed sufficiently small, there exists $\eps_0>0$ such that
$\tau_{\eps,N}(\xi)\equiv N$ for $|x|<\delta$ and $\eps<\eps_0$.
Also, it follows from smoothness of the flow and boundedness of first hit times for $|x|\ge\delta$ that $\tau_{\eps,N}\to\tau_{0,N}$ uniformly
on $\{|x|\ge\delta\}$.  Hence
$\lim_{\eps\to0}\|\tau_{\eps,N}-\tau_{0,N}\|_\infty=0$ and so 
$\lim_{\eps\to0}\|(I)\|_\infty=0$.

By continuity of the flow in initial conditions and parameters,
$X_\eps(\xi,t)\to X_0(\xi,t)$ uniformly in $\xi\in\Sigma$ and $t\in[0,N]$.
Hence $\lim_{\eps\to0}\|(II)\|_\infty=0$.

We have shown that $\limsup_{\eps\to0}\int|\Psi_\eps-\Psi_0|\,d\mu_{F_\eps} \le
2C'e^{-N/C}$.
The result follows since $N$ is arbitrary.
\end{proof}

Next, we show that $\sup_\eps\|\Psi_\eps\|_\alpha<\infty$ for some $\alpha>0$.

\begin{lemma} \label{lem:Psi}
Define $\widetilde\Psi_\eps:\Sigma\to\R$ by setting
$$\widetilde\Psi_\eps(x,y)=\int_0^{-\frac{1}{\lambda_{1,\eps}}\log |x|}
\tilde\psi(xe^{\lambda_{1,\eps}t},ye^{\lambda_{2,\eps}t},e^{\lambda_{3,\eps}t})\,dt.$$
Choose $0<\alpha'<-\lambda_{3,\eps}\beta/(\lambda_{1,\eps}-\lambda_{3,\eps})$.
Then $\widetilde\Psi_\eps$ is piecewise $C^{\alpha'}$.
Moreover, there is a constant $C>0$ such that $\|\widetilde\Psi_\eps\|_{\alpha'}\le C\|\psi\|_\beta$ for all $\eps\ge0$.
\end{lemma}

\begin{proof}
We suppress the dependence on $\eps$.
Write $\tilde\psi(x,y,z)=\psi_1(x)+\psi_2(x,y,z)$ where $\psi_1(x)=\tilde\psi(x,0,0)$.
Then $\psi_1$ and $\psi_2$ are H\"older with 
$H_\beta(\psi_1)\le H_\beta(\psi)$ and 
$H_\beta(\psi_2)\le 2H_\beta(\psi)$.
Also, $\psi_1(0)=0$ and $\psi_2(x,0,0)\equiv0$.
Define 
\[
\widetilde\Psi_1(x)=\int_0^{-\frac{1}{\lambda_1}\log x}
\psi_1(xe^{\lambda_1t})\,dt, \quad
\widetilde\Psi_2(x,y)=\int_0^{-\frac{1}{\lambda_1}\log x}
\psi_2(xe^{\lambda_1t},ye^{\lambda_2t},e^{\lambda_3t})\,dt.
\]
Recall that $\lambda_2<\lambda_3<0<\lambda_1$.

First we carry out the estimates for $\widetilde\Psi_1$ with $\alpha'=\beta$. 
By the change of variables $u=xe^{\lambda_1t}$,
\[
\widetilde\Psi_1(x)=\frac{1}{\lambda_1}\int_x^1 \frac{\psi_1(u)}{u}\,du.
\]
Now $|\psi_1(u)|=|\psi_1(u)-\psi_1(0)|\le H_\beta(\psi) u^\beta$, 
so $|\widetilde\Psi_1|_\infty \le \frac{H_\beta(\psi)}{\lambda_1} \int_0^1 u^{-(1-\beta)}\,du =\frac{H_\beta(\psi)}{\lambda_1\beta}$.
Also for $x_1>x_2>0$,
\[
\begin{aligned}
|\widetilde\Psi_1(x_1)-\widetilde\Psi_1(x_2)|&\le \frac{H_\beta(\psi)}{\lambda_1}\int_{x_2}^{x_1}u^{-(1-\beta)}\,du \\ 
&=\frac{H_\beta(\psi)}{\lambda_1\beta} (x_1^\beta-x_2^\beta) \le \frac{H_\beta(\psi)}{\lambda_1\beta}  (x_1-x_2)^\beta.
\end{aligned}
\]
Here, we have used that $x_1^\beta-x_2^\beta \le (x_1-x_2)^\beta$ for all $\beta\in[0,1]$.
Hence $\|\widetilde\Psi_1\|_\beta\le \frac{1}{\lambda_1\beta}\|\psi\|_\beta$.

Next, we carry out the estimates for $\widetilde\Psi_2$.  Note that
\begin{align}\ \label{eq:psi1} \nonumber
 |\psi_2(xe^{\lambda_1t},ye^{\lambda_2t},e^{\lambda_3t})|  = &
|\psi_2(xe^{\lambda_1t},ye^{\lambda_2t},e^{\lambda_3t})-\psi_2(xe^{\lambda_1 t},0,0)| \\
\nonumber
=& |\psi_2(xe^{\lambda_1t},ye^{\lambda_2t},e^{\lambda_3t})-\psi_2(xe^{\lambda_1t},ye^{\lambda_2t},0)| \\
\nonumber &+ |\psi_2(xe^{\lambda_1t},ye^{\lambda_2t},0)-\psi_2(xe^{\lambda_1t},0,0)| \\
\le & H_\beta(\psi_2)(e^{\beta\lambda_3 t}+|y|^\beta e^{\beta\lambda_2 t})\le 4 H_\beta(\psi) e^{\beta\lambda_3 t}.
\end{align}
In particular,
\[ 
|\widetilde\Psi_2|_\infty \le 4H_\beta(\psi) \int_0^{-\frac{1}{\lambda_1}\log x} e^{\beta\lambda_3 t}\,dt 
=\frac{4H_\beta(\psi)}{\beta|\lambda_3|}  (1-x^{-\beta\lambda_3/\lambda_1})\le  \frac{4H_\beta(\psi)}{\beta|\lambda_3|}.
\]
Similarly,
\[
|\widetilde\Psi_2(x,y_1)-\widetilde\Psi_2(x,y_2)|\le \int_0^{-\frac{1}{\lambda_1}\log x}H_\beta(\psi_2)|y_1-y_2|^\beta e^{\beta\lambda_2t}\,dt\le \frac{2H_\beta(\psi)}{\beta|\lambda_2|} |y_1-y_2|^\beta.
\]

For $x_1>x_2>0$, we have 
$|\widetilde\Psi_2(x_1,y)-\widetilde\Psi_2(x_2,y)|\le A+B$ where
\begin{align*}
A& = \int_{-\frac{1}{\lambda_1}\log x_1}^{-\frac{1}{\lambda_1}\log x_2}\psi_2(x_2e^{\lambda_1t},ye^{\lambda_2t},e^{\lambda_3t})\,dt, \\
B& = \int_0^{-\frac{1}{\lambda_1}\log x_1}|\psi_2(x_1e^{\lambda_1t},ye^{\lambda_2t},e^{\lambda_3t})-\psi_2(x_2e^{\lambda_1t},ye^{\lambda_2t},e^{\lambda_3t})|\,dt.
\end{align*}
By~\eqref{eq:psi1},
\[
A\le 4 H_\beta(\psi)  \int_{-\frac{1}{\lambda_1}\log x_1}^{-\frac{1}{\lambda_1}\log x_2} e^{\beta\lambda_3t}\,dt
=\frac{4H_\beta(\psi)}{\beta|\lambda_3|}(x_1^{\gamma}-x_2^{\gamma})\le \frac{4H_\beta(\psi)}{\beta|\lambda_3|} (x_1-x_2)^{\gamma}
\]
where $\gamma=-\lambda_3\beta/\lambda_1$.
Let 
\[
\alpha'=\delta\beta, \qquad
\lambda'=\delta \lambda_1+(1-\delta)\lambda_3,
\]
where $0<\delta<-\lambda_3/(\lambda_1-\lambda_3)<1$.
In particular, $\alpha'<\gamma$ and $\lambda'<0$.
Since the eigenvalues $\lambda_1,\lambda_3$ depend continuously on $\eps$, we can choose $\delta$ independent of $\eps$.
The inequality $\min\{a,b\}\le a^\delta b^{1-\delta}$ holds for all $a,b\ge0$.
By~\eqref{eq:psi1},
\begin{align*}
&|\psi_2(x_1e^{\lambda_1t},ye^{\lambda_2t},e^{\lambda_3t})-\psi_2(x_2e^{\lambda_1t},ye^{\lambda_2t},e^{\lambda_3t})|\\
 &\le \min\{H_\beta(\psi_2)(x_1-x_2)^\beta e^{\beta\lambda_1 t}, 8H_\beta(\psi) e^{\beta\lambda_3 t}\}
 \le 8H_\beta(\psi) (x_1-x_2)^{\delta\beta} e^{\delta\lambda' t}.
\end{align*}
Hence $B\le \frac{8}{\delta|\lambda'|} H_\beta(\psi) (x_1-x_2)^{\alpha'}$ completing the proof.
\end{proof}

\begin{corollary}\label{cor:reg}
There exists $\alpha\in(0,\beta)$ such that
$\Psi_\eps$ is piecewise $C^{\alpha}$.
Moreover, there is a constant $C>0$ such that $\sup_\eps\|\Psi_\eps\|_{\alpha}\le C\|\psi\|_\beta$.
\end{corollary}

\begin{proof}
After a $C^\alpha$ change of coordinates  in a neighbourhood of the singularity,
we may suppose without loss of generality that the flow $X_\eps$ is linear near the singularity.   Here, $\alpha>0$ can be chosen independent of $\eps$ with bounded H\"older constants for the linearisation~\cite{BarreiraValls07}.
We choose $\alpha<\alpha'$ where $\alpha'$ is as in Lemma~\ref{lem:Psi}

For the remainder of the proof we suppress the dependence on $\eps$.
The first hit time  in this neighbourhood is given by
$\tau_1(\xi)=-\frac{1}{\lambda_1}\log |x|$.  
Recall that $\xi=(x,y,1)$ and
set $\xi'=X(\xi,\tau_1(\xi))=(1,yx^{-\lambda_2/\lambda_1},x^{-\lambda_3/\lambda_1})$.  Note that the dependence of $\xi'$ on $\xi$ is H\"older (also uniformly in $\eps$).  Then
$\tau=\tau_1+\tau_2$ where $\tau_2$ is $C^\alpha$ in $\xi'$ and hence $\xi$ (uniformly in $\eps$).  Also,
$\Psi=\widetilde\Psi+\widehat\Psi$ where
\begin{equation}\label{eq:psi'}
\widetilde\Psi(\xi)=\int_0^{\tau_1(\xi)}\tilde\psi(xe^{\lambda_1 t},ye^{\lambda_2 t},e^{\lambda_3 t})\,dt, \quad
\widehat\Psi(\xi)=\int_0^{\tau_2(\xi)}\tilde\psi(X(\xi',t))\,dt.
\end{equation}
Since $\|\tau_2\|_\infty$ is bounded, it is
immediate that $\widehat\Psi$ is $C^\alpha$ uniformly in $\eps$.
By Lemma~\ref{lem:Psi}, the same is true for $\widetilde\Psi$.
\end{proof}

We can now state and prove the analogue of Theorem~\ref{main3} for normalised families.  Define $\|\psi\|_{(\eps)}=\int|\psi|\,d\mu_\eps+|\psi(0)|$.

\begin{corollary}\label{maincor} 
Let $X_\eps$ be a normalised family of flows admitting geometric Lorenz attractors.  Then
\begin{itemize}
\item[(a)]
$\lim_{\eps\to0}\sigma^2_{X_\eps}(\psi)=\sigma^2_{X_0}(\psi)$ 
for all $C^\beta$ observables  $\psi:\R^3\to\R$ 
\item[(b)]
For any $H>0$, 
there exists $C>0$  such that
\[
|\sigma^2_{X_\eps}(\psi)-\sigma^2_{X_\eps}(\psi')|\le C\|\psi-\psi'\|_{(\eps)}(1+|\log\|\psi-\psi'\|_{(\eps)}|)
\]
for all $C^\beta$ observables  $\psi,\,\psi':\R^3\to\R$ with
$\|\psi\|_\beta,\,\|\psi'\|_\beta\le H$, and all $\eps\ge0$.
\end{itemize}
\end{corollary}

 \begin{proof} 
By Corollary~\ref{cor:reg}, there exists  $C_1>0$, $\alpha\in(0,1)$ such that
$\Psi_\eps$, $\Psi'_\eps$ are piecewise $C^\alpha$ and
\begin{equation} \label{eq:C1}
\sup_\eps\|\Psi_\eps\|_\alpha\le C_1\|\psi\|_\beta, \qquad
\sup_\eps\|\Psi'_\eps\|_\alpha\le C_1\|\psi'\|_\beta.
\end{equation}
By~\eqref{eq:MT1} and Proposition~\ref{MelbTor},
\begin{equation*}\begin{aligned}
& |\sigma^2_{X_\eps}(\psi)-\sigma^2_{X_{\eps'}}(\psi')|
 =|\sigma^2_{X_\eps}(\tilde\psi)-\sigma^2_{X_{\eps'}}(\tilde\psi')|
 =\left| \frac{\sigma_{F_\eps}^2( \Psi_\eps)}{\int \tau_\eps\, d\mu_{F_\eps}}- \frac{\sigma_{F_{\eps'}}^2( \Psi'_{\eps'})}{\int \tau_{\eps'}\, d\mu_{F_{\eps'}}}\right|\\ 
&\qquad \le \sigma_{F_{\eps'}}^2( \Psi'_{\eps'}) \left| \frac{1}{\int \tau_\eps\, d\mu_{F_\eps}}- \frac{1}{\int \tau_{\eps'}\, d\mu_{F_{\eps'}}}\right|  +\frac{1}{\int \tau_\eps\, d\mu_{F_\eps}} \left|  \sigma_{F_\eps}^2( \Psi_{\eps})- \sigma_{F_{\eps'}}^2( \Psi'_{\eps'}) \right|.
\end{aligned}\end{equation*}

First suppose that $\psi=\psi'$ and $\eps'=0$.
By Lemma \ref{Modified_New_Marks} and~\eqref{eq:C1}, we can apply Theorem \ref{main2} to deduce that
$\lim_{\eps\to0}\sigma_{F_\eps}^2( \Psi_\eps)= \sigma_{F_0}^2( \Psi_0)$.
Part~(a) now follows from~\eqref{tauntau}.

Next suppose that $\eps=\eps'$.
By Proposition~\ref{p:main2} and~\eqref{eq:C1}, there is a constant $C>0$ (independent of $\eps$) such that
$|  \sigma_{F_\eps}^2( \Psi_\eps)- \sigma_{F_\eps}^2( \Psi'_\eps) |\le
C \|\Psi_\eps-\Psi'_\eps\|_{1,\eps}
(1+|\log \|\Psi_\eps-\Psi'_\eps\|_{1,\eps}|)$
where
\begin{equation*} \begin{aligned}
& \|\Psi_\eps-\Psi'_\eps\|_{1,\eps}
 =\int |\Psi_\eps-\Psi'_\eps|\,d\mu_{F_\eps}
 \le \int\int_0^{\tau_\eps(\xi)}|\tilde\psi(X_\eps(\xi,t)-\tilde\psi'(X_\eps(\xi,t)|\,dt\,d\mu_{F_\eps}
\\ & = \int\tau_\eps\, d\mu_{F_\eps}\int|\tilde\psi-\tilde\psi'|\,d\mu_\eps
 \le \int\tau_\eps\, d\mu_{F_\eps}\Big(\int|\psi-\psi'|\,d\mu_\eps+|\psi(0)-\psi'(0)|\Big).
\end{aligned}\end{equation*}
This proves part~(b).
 \end{proof}

 \begin{proof} [Proof of Theorem \ref{main3}]
Let 
$\phi_\eps:U\to U$ be the family of normalising conjugacies in section~\ref{sec:normal}.
Recall that
$\tilde X_\eps(x,t)=\phi_\eps^{-1}\circ X_\eps(\phi_\eps(x),t)$.
Define
\[
\psi_\eps=\psi\circ\phi_\eps,
\quad \psi'_\eps=\psi'\circ\phi_\eps,
\quad \tilde\mu_\eps={\phi_\eps^{-1}}\!_*\mu_\eps.
\]
Then $\psi_\eps$ is H\"older and
\[
\int\psi_\eps\,d\tilde\mu_\eps=
\int\psi\,d\mu_\eps, \qquad
\int_0^t \psi_\eps\circ \tilde X_\eps(s)\,ds=
\Big(\int_0^t \psi\circ X_\eps(s)\,ds\Big)\circ \phi_\eps.
\]
Hence $\int_0^t \psi\circ X_\eps(s)\,ds-t\int\psi\,d\mu_\eps$ 
and 
$\int_0^t \psi_\eps\circ \tilde X_\eps(s)\,ds-t\int\psi_\eps\,d\tilde\mu_\eps$ have the same distribution (relative to the probability measures $\mu_\eps$ and $\tilde\mu_\eps$ respectively) 
so it follows from~\eqref{eq:CLT} that $\sigma^2_{X_\eps}(\psi)=\sigma^2_{\tilde X_\eps}(\psi_\eps)$.
Similar comments apply to $\psi'$.

It follows from the definitions that $\tilde X_\eps$ is a normalised family.
By Corollary~\ref{phieps}, there exists $H>0$ and $\beta\in(0,1)$ such that
$\sup_\eps\|\psi_\eps\|_\beta\le H$ and 
$\sup_\eps\|\psi'_\eps\|_\beta\le H$.

By Corollary~\ref{maincor}(b), there is a constant $C>0$ such that
\[
\begin{aligned}
|\sigma_{X_\eps}^2(\psi)- & \sigma_{X_0}^2(\psi)|
 =|\sigma_{\tilde X_\eps}^2(\psi_\eps)-\sigma_{\tilde X_0}^2(\psi_0)|
\\ &  \le |\sigma_{\tilde X_\eps}^2(\psi_\eps)-\sigma_{\tilde X_\eps}^2(\psi_0)|
+|\sigma_{\tilde X_\eps}^2(\psi_0)-\sigma_{\tilde X_0}^2(\psi_0)|
\\ &  \le C\|\psi_\eps-\psi_0\|_\infty(1+|\log\|\psi_\eps-\psi_0\|_\infty|)
+|\sigma_{\tilde X_\eps}^2(\psi_0)-\sigma_{\tilde X_0}^2(\psi_0)|.
\end{aligned}
\]
By Corollary~\ref{maincor}(a),
the second term on the right-hand side converges to zero as $\eps\to0$.
Also, $\lim_{\eps\to0}\|\psi_\eps-\psi_0\|_\infty=\lim_{\eps\to0}\|\psi\circ \phi_\eps^{-1}-\psi\circ\phi_0^{-1}\|_\infty=0$ by
Corollary~\ref{phieps}.
Hence $\lim_{\eps\to0} \sigma_{X_\eps}^2(\psi)=\sigma_{X_0}^2(\psi)$ proving part~(a).

Finally, by Corollary~\ref{maincor}(b),
\[
\begin{aligned}
 |\sigma^2_{X_0}(\psi)-\sigma^2_{X_0}(\psi')
& =|\sigma^2_{\tilde X_0}(\psi_0)-\sigma^2_{\tilde X_0}(\psi'_0)|
\\ & \le C\|\psi_0-\psi'_0\|_{(0)}(1+|\log\|\psi_0-\psi'_0\|_{(0)}|)
 \\ & = C\|\psi-\psi'\|(1+|\log\|\psi-\psi'\||)
\end{aligned}
\]
yielding part~(b).
\end{proof}

\end{document}